\documentclass[11pt]{article}

\usepackage[margin=1in]{geometry}
\usepackage{amsmath,amssymb,amsthm,mathtools,bm}
\usepackage{enumitem}
\usepackage{hyperref}
\usepackage[T1]{fontenc}
\usepackage{lmodern}
\usepackage{microtype}

\hypersetup{
  colorlinks=true,
  linkcolor=blue,
  citecolor=blue,
  urlcolor=blue
}

\newcommand\blfootnote[1]{
  \begingroup
  \renewcommand\thefootnote{}\footnote{#1}
  \addtocounter{footnote}{-1}
  \endgroup
}

\newtheorem{theorem}{Theorem}[section]
\newtheorem{proposition}[theorem]{Proposition}
\newtheorem{lemma}[theorem]{Lemma}
\newtheorem{corollary}[theorem]{Corollary}
\theoremstyle{definition}

\newtheorem{remark}[theorem]{Remark}
\newtheorem{example}[theorem]{Example}

\numberwithin{equation}{section}

\def \kCay{\operatorname{k-Cay}}
\def \Z{\mathbb{Z}}
\def \C{\mathbb{C}}
\def \R{\mathbb{R}}
\def \y{{\bm y}}

\title{Fourier Decomposition and E-Eigenvalues of Hypergraph Polynomials over Finite Abelian Groups}
\author{\normalsize{Yan-Hong Bao, Yi-Zheng Fan, Lili Liao*}\\
\footnotesize{School of Mathematical Sciences, Anhui University, Hefei 230601, China }}
\date{}
\begin{document}
\maketitle
\blfootnote{*Corresponding author. Email: baoyh@ahu.edu.cn, fanyz@ahu.edu.cn, lililiao@stu.ahu.edu.cn.}

\begin{abstract}
Let $H=(V,E)$ be a $k$-uniform hypergraph with polynomial
$f_H(\bm{x})=\sum_{e\in E}\prod_{v\in e}x_v$.
For a finite abelian group $G$ and a subgroup $S$ of $G$, the
$k$-partite Cayley hypergraph $K$ has $k$ disjoint vertex
parts, each indexed by $G$, and $\{(1,g_1),\ldots,(k,g_k)\}$ is an
edge if and only if $g_1+\cdots+g_k\in S$.
For $k=2$, $K$ is isomorphic to the Bi-Cayley graph associated with
$G$ and $S$.
For $k\ge3$, we use character theory to express $f_K$, under a complex
unitary change of variables, as a sum of exactly $[G:S]$ nonzero
monomials in pairwise disjoint sets of variables. This yields a direct
sum decomposition whenever $S$ is a proper subgroup of $G$.
When $G/S$ is an elementary abelian $2$-group, including the case
$G=\Z_2^m$, this decomposition can be realized by a real orthogonal
change of variables, enabling us to determine the complete sets of complex E-eigenvalues
and Z-eigenvalues of the adjacency tensor of $K$. 
\end{abstract}


\section{Introduction}

Cayley graphs are fundamental objects in algebraic graph theory, providing a natural link between the algebraic structure of a group and the combinatorial and spectral properties of a graph \cite{Babai1979,Biggs1993,GodsilRoyle2001}. For a finite group $G$ and an inverse-closed subset $S\subseteq G\setminus\{1_G\}$, the Cayley graph $\operatorname{Cay}(G,S)$ is defined on the vertex set $G$, with two vertices $g,h\in G$ joined by an edge whenever $hg^{-1}\in S$. When $G$ is abelian and written additively, the condition $S=S^{-1}$ becomes $S=-S$, and its spectral structure admits an explicit description. The characters of $G$ diagonalize the adjacency matrix, and the corresponding eigenvalues are
\[
\lambda_\chi=\sum_{s\in S}\chi(s),\quad \chi\in\widehat G,
\]
where $\widehat G$ denotes the character group of $G$.

Bi-Cayley graphs, studied for example in~\cite{LuWangXu2006,ZouMeng2007}, extend this construction to a bipartite setting. For a finite group $G$ and a subset $S\subseteq G$, the Bi-Cayley graph $\operatorname{BiCay}(G,S)$ is the bipartite graph with vertex set $G\times \{0,1\}$ and edge set $\bigl\{\{(g,0),(sg,1)\}:g\in G,\ s\in S\bigr\}$. When $G$ is abelian, its spectrum also admits a character-theoretic description analogous to that of Cayley graphs. These results illustrate the effectiveness of group characters in converting algebraically defined adjacency relations into tractable spectral information.

Motivated by this observation, let $G$ be a finite abelian group written additively, $S$ a subgroup of $G$, and $k\geq3$. We study the $k$-partite $k$-uniform hypergraph $K$ whose edges $\{(1,g_1),\ldots,(k,g_k)\}$ are defined by
\[
g_1+\cdots+g_k\in S,
\]
where $g_i\in G$ for each $i\in[k]$ and the first coordinate specifies the vertex part. We call $K$ the \emph{$k$-partite Cayley hypergraph associated with $G$ and $S$}, and write $K=\kCay(G,S)$. For $k=2$, the same construction gives a graph isomorphic to $\operatorname{BiCay}(G,S)$. Thus, this construction extends the additive relation underlying Bi-Cayley graphs to the $k$-uniform hypergraph setting while retaining the group structure required for Fourier analysis. Note that the $k$-partite Cayley hypergraph defined here differs from the Cayley hypergraphs considered in \cite{Buratti1994}.

To study its algebraic structure, we consider the associated hypergraph polynomial
\[
f_K(\bm{x})
=
\sum_{g_1+\cdots+g_k\in S}
x_{1,g_1}x_{2,g_2}\cdots x_{k,g_k}.
\]
To analyze this hypergraph polynomial, we apply Fourier analysis on finite abelian groups. Let $\widehat G=\operatorname{Hom}(G,\C^\times)$ be the character group of $G$, and let
\[
S^\perp
=
\{\chi\in\widehat G :\chi(s)=1\text{ for all }s\in S\}
\]
be the annihilator subgroup of $S$. By expanding the subgroup indicator
$\mathbf 1_S$ in characters, we obtain a Fourier decomposition of $f_K(\bm{x})$.
Under the normalized Fourier transformation, $f_K(\bm{x})$ becomes the sum
of $|S^\perp|$ monomials, with a common nonzero coefficient, in pairwise disjoint sets of variables. When
$|S^\perp|\geq2$, this gives a direct sum decomposition. Since
$|S^\perp|=[G:S]$, the subgroup index determines the number of nonzero summands, thereby relating the subgroup structure of $G$ to the direct sum structure of the
associated hypergraph polynomial. For a general finite
abelian group, the decomposition is given by a complex unitary Fourier
transformation. If $G/S$ is an elementary abelian $2$-group, the characters
in $S^\perp$ are real-valued, and their normalized coordinate vectors can
be completed to a real orthonormal basis. This yields the same decomposition
under a real orthogonal change of variables. When $G=\Z_2^m$, one may
use the full real Walsh--Hadamard orthogonal transformation.

This orthogonal case is particularly relevant to tensor spectral theory. Tensor eigenvalue theory, initiated by Qi and Lim \cite{Qi2005,Lim2005} and developed for hypergraphs by Cooper--Dutle, Pearson--Zhang, and others \cite{CooperDutle2012,PearsonZhang2014}, provides a spectral framework for uniform hypergraphs. E-eigenvalues and Z-eigenvalues have been studied from
several perspectives and applied in various contexts. For example, related studies include the E-characteristic polynomial of
even-order tensors \cite{NiQiWangWang2007}, Z-eigenvalues and their
applications to spectral hypergraph theory \cite{LiQiYu2013},
Z-eigenvalue localization with applications to multipartite quantum
entanglement \cite{XiongJiangLiuQin2022}, and the best rank-one
approximation of symmetric tensors \cite{GuanChuChu2018}.
However, it is generally difficult to determine all E-eigenvalues of a hypergraph explicitly. When $G/S$ is an elementary abelian $2$-group, the real orthogonal change of variables described above expresses the hypergraph
polynomial as a sum of monomials in disjoint sets of variables, giving a direct sum
whenever $|S^\perp|\ge2$. This enables us to compute the E-eigenvalues explicitly. In
particular, we determine the complete set of complex E-eigenvalues in terms of sums of $(k-2)$th roots of unity and derive an explicit description of the Z-eigenvalues. Consequently, for fixed $G$ and $k$, among subgroups $S$ for which $G/S$ is an elementary abelian $2$-group, the subgroup index
determines not only the number of summands in the Fourier decomposition but
also the resulting E-eigenvalue set.

The paper is organized as follows. In Section~\ref{sec:preliminaries}, we recall some basic notions and notations, including $k$-partite Cayley hypergraphs, tensor eigenvalues, and polynomial decompositions. Section~\ref{sec:fourier} develops the Fourier decomposition over
finite abelian groups and determines the transformed adjacency tensor.
Section~\ref{sec:eigenvalues} treats the case where $G/S$ is an elementary
abelian $2$-group and derives the complex E-eigenvalues and Z-eigenvalues
of the adjacency tensor, with $G=\Z_2^m$ as a special case.

\section{Preliminaries}\label{sec:preliminaries}

\subsection{\texorpdfstring{The $k$-partite Cayley hypergraph and its polynomial}{The k-partite Cayley hypergraph and its polynomial}}

A \emph{hypergraph} $H=(V(H),E(H))$ consists of a vertex set $V(H)$ and an
edge set $E(H)$ whose elements are subsets of $V(H)$. The \emph{degree} of a
vertex $v\in V(H)$, denoted by $d_v$, is the number of edges of $H$
that contain $v$. If every edge $e\in E(H)$ satisfies $|e|=k$, then $H$
is called a \emph{$k$-uniform hypergraph}.
A $k$-uniform hypergraph is called \emph{$k$-partite} if its vertex set has a partition
\[
  V=V_1 \sqcup V_2\sqcup\cdots\sqcup V_k
\]
such that every edge contains exactly one vertex from each part.
A hypergraph is called \emph{linear} if every two distinct edges have at most one vertex in common.

Let $H$ be a $k$-uniform hypergraph on vertex set
$[n]:=\{1,\ldots,n\}$, with edge set $E(H)$. For an edge
$e=\{i_1,\ldots,i_k\} \in E(H)$, we denote by $x^e$
 the monomial $x_{i_1}x_{i_2}\cdots x_{i_k}$.
The \emph{hypergraph polynomial}
$f_H(\bm{x})$ of $H$ is defined by
\begin{align}\label{2.1-eq-hypergraph ploy}
f_H(\bm{x})=\sum_{{e}\in E(H)}x^e.
\end{align}

Let $G$ be a finite abelian group written additively, $S$ a subgroup of $G$, and $k\geq 2$.
We define a hypergraph $K$ associated with $G$ and $S$ as follows: the vertex set is $V=V_1\sqcup V_2\sqcup \cdots \sqcup V_k$ with
\[
  V_i=\{(i,g):g\in G\}, \quad i\in [k],
\]
and a set $\{(1,g_1),(2,g_2),\ldots,(k,g_k)\}$ is an edge of $K$ if and only if
\[
  g_1+g_2+\cdots+g_k\in S.
\]
Clearly, $K$ is a $k$-partite $k$-uniform hypergraph, which is called the \emph{$k$-partite Cayley hypergraph associated with $G$ and $S$}. We write $K=\kCay(G,S)$.

We attach a variable $x_{i,g}$ to the vertex $(i,g)$. The associated hypergraph polynomial of $K$ is
\begin{equation}\label{eq:hypergraph-polynomial}
  f_K(\bm{x})=\sum_{g_1+\cdots+g_k\in S} x_{1,g_1}x_{2,g_2}\cdots x_{k,g_k}.
\end{equation}

\begin{proposition}\label{prop:basic}
Retain the above notation. Then
\begin{enumerate}[label=\textup{(\arabic*)}]
  \item $|V(K)|=k|G|$ and $|E(K)|=|S|\,|G|^{k-1}$.
\item When $k=2$, the hypergraph $K$ is isomorphic to the
Bi-Cayley graph $\operatorname{BiCay}(G,S)$.
    \item If $S=\{0\}$ and $k=3$, then $K$ is linear. If $k\geq 4$ and $|G|>1$, then $K$ is not linear.
\end{enumerate}
\end{proposition}

\begin{proof}
 Part \textup{(1)} follows immediately from the definition of $k$-partite Cayley hypergraph.

(2) Recall that the Bi-Cayley graph $\operatorname{BiCay}(G,S)$ has the vertex set $G\times\{0,1\}$ and the edge set
\[
\bigl\{\{(g,0),(s+g,1)\}:g\in G,\ s\in S\bigr\}.
\]
For $k=2$, the edge condition for the $k$-partite Cayley hypergraph is $g_1+g_2\in S$. We consider a mapping
$\sigma\colon V(K)\to V(\operatorname{BiCay}(G,S))$ given by
\[\sigma((1,g_1))=(g_1,0) \quad {\rm and}\quad \sigma((2,g_2))=(-g_2,1).\]

Clearly, $\sigma$ is a bijection.
Since $S$ is a subgroup of $G$, $g_1+g_2\in S$ if and only if there exists $s\in S$ such that $-g_2=g_1+s$.
Hence, $\{(1, g_1), (2, g_2)\}$ is an edge in $K$ if and only if $\{\sigma((1, g_1)), \sigma((2, g_2))\}$ is an edge in $\operatorname{BiCay}(G,S)$. Therefore, the mapping $\sigma$ is an isomorphism from $K$ to $\operatorname{BiCay}(G,S)$.

(3) If $S=\{0\}$ and $k=3$, then $e=\{(1, g_1), (2, g_2), (3, g_3)\}$ is an edge in $K$ if and only if
\[g_1+g_2+g_3=0.\]
Suppose that $e'=\{(1, g'_1), (2, g'_2), (3, g'_3)\}$ is also an edge. Clearly, if $|e\cap e'|\ge 2$, then $e=e'$, and therefore $K$ is linear.

If $k\geq4$ and $|G|>1$, we assume that $e=\{(1, g_1), \cdots, (k, g_k)\}$ is an edge in $K$.
Taking an arbitrary nonzero element $h\in G$, we write $g'_1=g_1+h$, $g'_2=g_2-h$ and $g'_i=g_i$ for all $3\le i\le k$. Clearly $g'_1+\cdots+g'_k=g_1+\cdots+g_k\in S$.  It follows that $e'=\{(1, g'_1), \cdots, (k, g'_k)\}\neq e$ is also an edge in $K$ and $|e\cap e'|=k-2\ge 2$. Hence $K$ is not linear.
\end{proof}

\subsection{E-eigenvalues of hypergraphs}
We recall some basic notions for tensors, following Qi \cite{Qi2005}.
Let $\Bbbk$ be either $\R$ or $\C$.
An \emph{order-$m$, dimension-$n$ tensor} $\mathcal{A}=(a_{i_1 i_2 \cdots i_m})$ consists of
$n^m$ entries $a_{i_1 i_2\cdots i_m}\in\Bbbk$,
where $i_j \in [n]$ for each $j\in [m]$. The tensor $\mathcal A$ is
said to be \emph{supersymmetric} if its entries are invariant under any permutation
of their indices \cite{KofidisRegalia2002}.

For an order-$m$, dimension-$n$ tensor $\mathcal{A}=(a_{i_1 i_2 \cdots i_m})$, we introduce a homogeneous polynomial of degree $m$
\[
f(\bm{x}):=\mathcal A\bm{x}^m
=
\sum_{i_1,i_2, \ldots,i_m=1}^n a_{i_1,i_2,\cdots,i_m}x_{i_1}\cdots x_{i_m},
\]
where $\bm{x}=(x_1,\ldots,x_n)^\top$.

A complex number $\lambda$ is called an \emph{E-eigenvalue} of $\mathcal A$ if
there exists a nonzero vector $\bm{x}\in\C^n$ such that
\begin{equation}\label{eq:E-eigen}
    \mathcal A\bm{x}^{m-1}=\lambda\bm{x},
\quad
\bm{x}^\top \bm{x}=1,
\end{equation}
where $\mathcal{A}\bm{x}^{m-1}$ is an $n$-dimensional vector given by
\[(\mathcal{A}\bm{x}^{m-1})_i=\sum_{1\le i_2, \cdots,i_m\le n}a_{ii_2\cdots i_m}x_{i_2}\cdots x_{i_m},\]
and $\bm{x}^\top\bm{x}=\sum_i x_i^2$. The vector $\bm{x}$ is called an
\emph{E-eigenvector} associated with $\lambda$. If both $\bm{x}$ and $\lambda$ are real, the eigenpair $(\lambda, \bm{x})$ is
called a real \emph{E-eigenpair}, also called a \emph{Z-eigenpair} in the literature \cite{Lim2005,Qi2005}. We denote by $\operatorname{Spec}_E(\mathcal A)$ the set of all complex
\emph{E-eigenvalues} of $\mathcal A$, and by $\operatorname{Spec}_Z(\mathcal A)$
the subset of \emph{Z-eigenvalues} of $\mathcal A$, namely those
real E-eigenvalues that admit real normalized eigenvectors. Throughout this paper, spectra are sets of eigenvalues; multiplicities are not included.

For an $n\times n$ complex matrix $P=(p_{ij})$, we consider the tensor $\mathcal B=P^m\mathcal A$ given by
\begin{equation}
b_{i_1i_2\cdots i_m}=\sum_{j_1,j_2,\dots,j_m=1}^n p_{i_1j_1}p_{i_2j_2}\cdots p_{i_mj_m}a_{j_1j_2\cdots j_m}.
\end{equation}
If $P$ is real and orthogonal, then $\mathcal A$ and $\mathcal B$ are said to be \emph{orthogonally similar}.
The following result plays a key role in our paper.

\begin{theorem}[\cite{Qi2005}, Theorem 7]
\label{thm:qi-orthogonal-invariance}
If supersymmetric tensors $\mathcal A$ and $\mathcal B$ are orthogonally similar, then they have the same E-eigenvalues. In particular, if $\mathcal B = P^m\mathcal A$, $\lambda$ is an E-eigenvalue of $\mathcal A$ and $\bm{x}$ is an E-eigenvector of $\mathcal A$ associated with $\lambda$, where $P$ is an $n\times n$ real orthogonal matrix, then $\lambda$ is also an E-eigenvalue of $\mathcal B$ and $\bm{y}=P\bm{x}$ is an E-eigenvector of $\mathcal B$ associated with $\lambda$.
\end{theorem}

Let $H$ be a $k$-uniform hypergraph on vertex set
$[n]$.
Cooper and Dutle \cite{CooperDutle2012} introduced the \emph{adjacency tensor}
$\mathcal A_H=(a_{i_1\cdots i_k})$ of $H$, which is the supersymmetric tensor of order
$k$ and dimension $n$ given by
\begin{equation}
a_{i_1,i_2,\cdots i_k}=
\begin{cases}
\dfrac{1}{(k-1)!}, & \{i_1,\ldots,i_k\}\in E(H),\\[6pt]
0, & \text{otherwise}.
\end{cases}\label{eq:adjacency-tensor}
\end{equation}
The E-eigenvalues (respectively, E-eigenvectors, E-eigenpairs) of $H$ refer to those of its
adjacency tensor.

Clearly, we have
\begin{equation}\label{eq:tensor-gradient}
  \mathcal A_H\bm{x}^{k-1}=\nabla f_H(\bm{x})
  \quad {\rm and} \quad
  \mathcal A_H\bm{x}^k=k f_H(\bm{x}),
\end{equation}
where $f_H(\bm{x})$ is the hypergraph polynomial of $H$ introduced in \eqref{2.1-eq-hypergraph ploy}, and $\nabla f_H(\bm{x})$ denotes the gradient of $f_H$, whose $i$-th
component is $\partial f_H/\partial x_i$. Consequently, a complex number $\lambda$ is an E-eigenvalue of
$H$ if and only if there exists a nonzero vector
$\bm{x}\in\C^n$ satisfying
\begin{equation}\label{eq:eigenvalue-lambda}
\nabla f_H(\bm{x})=\lambda\bm{x},\quad
\bm{x}^{\top}\bm{x}=1.
\end{equation}

\subsection{Direct sum decompositions of homogeneous polynomials}

Recall that a \emph{higher degree form} means a homogeneous polynomial of degree $d>2$.
In this subsection, let $d\ge 3$ be an integer, and let $\Bbbk$ be a field
with $\operatorname{char}\Bbbk=0$ or
$\operatorname{char}\Bbbk>d$. A form of degree $d$ in $n$ variables
is an element of the polynomial ring $\Bbbk[x_1,\ldots,x_n]$ which is
a sum of monomials of degree $d$. A form $f$ is called a \emph{direct sum} if,
after an invertible linear change of variables, it can be written as a sum
of $t\ge 2$ forms in disjoint sets of variables as follows:
\[
f = f_1(y_1,\ldots,y_{a_1}) + \cdots
+ f_t(y_{a_{t-1}+1},\ldots,y_n).
\]
If $f$ is not a direct sum, we say that $f$ is \emph{indecomposable}. For a detailed treatment of direct sum decompositions of higher degree
forms, we refer to \cite{HuangLvZhangYe2018}.

The center theory of higher degree forms provides a criterion for direct sum decompositions and diagonalization of higher degree forms \cite{Harrison1975,HuangLuYeZhang2021}. Observe that the decomposition
of higher degree forms can also be extended to a collection of multivariate polynomials and even to quadratic forms. For a collection of multivariate polynomials, their centers form a special Jordan algebra, which yields a criterion and an algorithm for simultaneous direct sum decompositions \cite{FangHuangLiao2025}. In the quadratic case, center theory can be used to handle the simultaneous block diagonalization of a set of symmetric matrices via congruence, and can be further extended to Hermitian matrices via \(*\)-congruence \cite{FangHuangHuang2026}. These results provide an algebraic framework for studying direct sum decompositions through the center theory.

\section{Fourier decomposition over finite abelian groups}\label{sec:fourier}

\subsection{Character theory of finite abelian groups}
We first recall some basic facts from the character theory of finite abelian groups, which will be used in the Fourier decomposition below. For an accessible introduction to the representation theory of finite groups and its character-theoretic aspects, we refer the reader to \cite{Steinberg2012}; for Fourier analysis on finite groups, see \cite{Luong2009,Terras1999}.

Let \(G\) be a finite abelian group, written additively. A \emph{character} of \(G\) is a group homomorphism from \(G\) to the multiplicative group \(\C^\times\) of nonzero complex numbers. The set of all such homomorphisms forms a group under pointwise multiplication, called the \emph{character group} of \(G\), denoted by $
\widehat G = \operatorname{Hom}(G, \C^\times).$
It is a standard fact that \(\widehat G \cong G\) for any finite abelian group \(G\); in particular, \(|\widehat G| = |G|\). Moreover, since $G$ is abelian, every irreducible complex representation of $G$ is one-dimensional. Thus, the characters of $G$, which are the group homomorphisms from $G$ to $\C^\times$, are precisely all the irreducible characters of $G$ in the sense of representation theory.

Let $\C^G$ denote the vector space of all complex-valued functions on $G$. For $\phi, \psi\in \C^G$, we consider the inner product
\[
\langle \phi, \psi \rangle = \frac{1}{|G|} \sum_{g \in G} \phi(g) \overline{\psi(g)}.
\]
Observe that the irreducible characters of \(G\) satisfy the orthogonality relation, that is, for
$\chi, \chi'\in \widehat{G}$,
\begin{equation}\label{eq:character-orthogonality}
\langle \chi, \chi' \rangle=\frac{1}{|G|} \sum_{g \in G} \chi(g) \overline{\chi'(g)} =
\begin{cases}
1, & \chi = \chi', \\
0, & \chi \neq \chi',
\end{cases}
\end{equation}
see \cite[Chapter~4]{Steinberg2012}. These are the classical Schur orthogonality relations restricted to the abelian setting. They imply that the irreducible characters of \(G\) form an orthonormal basis of \(\C^G\) with respect to the above inner product.
Consequently, every function \(\Phi: G \to \C\) admits a unique expansion in terms of irreducible characters. Indeed, for a function \(\Phi: G \to \C\), its Fourier transform $\widehat{\Phi}\colon \widehat{G}\to \C$ is defined by
\begin{equation}\label{eq:fourier-transform}
\widehat \Phi(\chi) \colon= \langle \Phi, \chi \rangle =\frac{1}{|G|} \sum_{g \in G} \Phi(g) \overline{\chi(g)}, \quad \chi \in \widehat G.
\end{equation}
Thus, the Fourier inversion formula asserts that for every \(g \in G\),
\begin{equation}\label{eq:fourier-inversion}
\Phi(g) = \sum_{\chi \in \widehat G} \widehat \Phi(\chi) \chi(g).
\end{equation}

Now, let \(S \leq G\) be a subgroup. Its \emph{annihilator} in the character group is defined as
\[
S^\perp = \{\chi \in \widehat G \mid \chi(s) = 1 \text{ for every } s \in S\}.
\]
A fundamental duality theorem for finite abelian groups gives $
S^\perp \cong \widehat{G/S},$
and, therefore,
\begin{equation}\label{eq:annihilator-size}
|S^\perp| = |G/S| = [G : S].
\end{equation}
Thus, the size of the annihilator is precisely the index of \(S\) in \(G\), a fact that will be essential in determining the number of nonzero summands in the Fourier decomposition below.

\begin{lemma}[Fourier expansion of subgroup indicators]\label{lem:subgroup-indicator}
Let $G$ be a finite abelian group and $S$ a subgroup of $G$. Suppose that $\mathbf{1}_S$ is the indicator function of $S$ given by $\mathbf{1}_S(g)=\begin{cases} 1, & g\in S, \\0, & g \in G\setminus S\end{cases}$. Then for every $g\in G$,
\begin{equation}\label{eq:subgroup-indicator}
\mathbf{1}_S(g)=\frac{|S|}{|G|}\sum_{\chi\in S^\perp}\chi(g).
\end{equation}
\end{lemma}

\begin{proof}
For any $\chi\in\widehat G$, by the definition of the Fourier transform \eqref{eq:fourier-transform}, we have
\[
\widehat{\mathbf{1}_S}(\chi)
=\frac{1}{|G|}\sum_{g\in G}\mathbf{1}_S(g)\overline{\chi(g)}
=\frac{1}{|G|}\sum_{s\in S}\overline{\chi(s)}.
\]
If $\chi\in S^\perp$, then $\chi(s)=1$ for every $s\in S$, and therefore
\[
\widehat{\mathbf{1}_S}(\chi)=\frac{|S|}{|G|}.
\]
If $\chi\notin S^\perp$, then $\chi|_S$ is a nontrivial character of $S$. Hence, there exists $s_0\in S$ such that $\chi(s_0)\neq1$. Denote
$
A=\sum_{s\in S}\overline{\chi(s)}
$.
Since $S$ is a subgroup, the translation $s\mapsto s+s_0$ is a bijection on $S$. Therefore
\[
A=\sum_{s\in S}\overline{\chi(s+s_0)}
=\overline{\chi(s_0)}A.
\]
By $\overline{\chi(s_0)}\neq1$, we have $A=0$ and $\widehat{\mathbf{1}_S}(\chi)=0$. Thus, we get
\[
\widehat{\mathbf{1}_S}(\chi)=
\begin{cases}
|S|/|G|, & \chi\in S^\perp,\\
0, & \chi\notin S^\perp.
\end{cases}
\]
Therefore, by the Fourier inversion formula \eqref{eq:fourier-inversion}, we have
\[
\mathbf{1}_S(g)=\sum_{\chi\in\widehat G}\widehat{\mathbf{1}_S}(\chi)\chi(g)=\frac{|S|}{|G|}\sum_{\chi\in S^\perp}\chi(g)
\]
for every $g\in G$.
\end{proof}

Lemma \ref{lem:subgroup-indicator} plays a central role in the proof of the Fourier decomposition theorem in the next subsection.

\subsection{Fourier decomposition of the hypergraph polynomial}

Replacing $g$ by $g_1 + \cdots + g_k$ in \eqref{eq:subgroup-indicator}, we obtain the identity
\begin{equation}\label{eq:additive-indicator}
\mathbf 1_{S}(g_1 + \cdots + g_k)
= \frac{|S|}{|G|} \sum_{\chi \in S^\perp} \chi(g_1) \chi(g_2) \cdots \chi(g_k).
\end{equation}

For $i\in [k]$ and $\chi \in \widehat G$, one can define the linear form
\begin{equation}\label{eq:linear-form}
L_{i,\chi} = \sum_{g \in G} \chi(g) x_{i,g}.
\end{equation}

\begin{theorem}[Fourier decomposition]\label{thm:fourier-decomposition}
Let $G$ be a finite abelian group, $S$ a subgroup of $G$, and $k\ge3$. Set $K=\kCay(G,S)$. Then the hypergraph polynomial $f_K(\bm{x})$ has the Fourier decomposition
\begin{equation}\label{eq:fourier-decomposition}
f_{K}(\bm{x}) = \frac{|S|}{|G|} \sum_{\chi \in S^\perp} L_{1,\chi} L_{2,\chi} \cdots L_{k,\chi}.
\end{equation}
\end{theorem}

\begin{proof}
Using \eqref{eq:additive-indicator}, we obtain
\begin{align*}
f_K(\bm{x})
&= \sum_{g_1, \ldots, g_k \in G}
   \mathbf 1_{S}(g_1 + \cdots + g_k)
   x_{1,g_1} x_{2,g_2} \cdots x_{k,g_k} \\
&= \frac{|S|}{|G|} \sum_{\chi \in S^\perp}
   \sum_{g_1, \ldots, g_k \in G}
   \chi(g_1)\chi(g_2) \cdots \chi(g_k) x_{1,g_1} x_{2,g_2} \cdots x_{k,g_k} \\
&= \frac{|S|}{|G|} \sum_{\chi \in S^\perp}
   \prod_{i=1}^k \left( \sum_{g \in G} \chi(g) x_{i,g} \right) \\
&= \frac{|S|}{|G|} \sum_{\chi \in S^\perp} L_{1,\chi} L_{2,\chi} \cdots L_{k,\chi}.
\end{align*}
\end{proof}

We consider the normalized Fourier coordinates by setting
\begin{equation}\label{eq:fourier-coordinates}
y_{i,\chi} = \frac{1}{\sqrt{|G|}} \sum_{g \in G} \chi(g) x_{i,g},
\quad i\in [k], \quad \chi \in \widehat G.
\end{equation}
By the orthogonality relation \eqref{eq:character-orthogonality}, the matrix \(U = (|G|^{-\frac{1}{2}} \chi(g))_{\chi\in\widehat G,\;g\in G}\) is unitary over \(\C\). Hence \eqref{eq:fourier-coordinates} gives an invertible complex-linear change of variables.
Set
\[
P=\operatorname{diag}(\underbrace{U,\ldots,U}_{k\text{ copies}}).
\]
Then $\bm y=P\bm x$ and $\bm x=P^*\bm y$, where $P^*$ denotes the conjugate transpose of $P$.

\begin{theorem}[Complex direct-sum decomposition]\label{thm:direct-sum}
Let $G$ be a finite abelian group, $S$ a subgroup of $G$, and $k\ge3$. Set $K=\kCay(G,S)$. Using the normalized Fourier coordinates \eqref{eq:fourier-coordinates}, the hypergraph polynomial $f_K(\bm{x})$ has the form
\begin{equation}\label{eq:normalized-direct-sum}
f_K(P^*\bm{y}) = c \sum_{\chi \in S^\perp} y_{1,\chi} y_{2,\chi} \cdots y_{k,\chi},
\quad
c = |S| \, |G|^{\frac{k-2}{2}}.
\end{equation}
The displayed decomposition has exactly $|S^\perp| = [G : S]$ nonzero monomial summands in pairwise disjoint sets of variables. Consequently, $f_K$ is a direct sum over $\C$ whenever $[G:S]\ge2$. The variables $y_{i,\chi}$ with $\chi \notin S^\perp$ do not occur in the transformed polynomial.
\end{theorem}

\begin{proof}
Since $L_{i,\chi} = \sqrt{|G|} \, y_{i,\chi}$, Theorem \ref{thm:fourier-decomposition} gives
\begin{align*}
    f_K(P^*\bm y)
&= \frac{|S|}{|G|} \sum_{\chi \in S^\perp}
  \left( \sqrt{|G|} \, y_{1,\chi} \right) \left( \sqrt{|G|} \, y_{2,\chi} \right) \cdots
  \left( \sqrt{|G|} \, y_{k,\chi} \right)\\
&= |S| \, |G|^{\frac{k-2}{2}} \sum_{\chi \in S^\perp} y_{1,\chi} y_{2,\chi} \cdots y_{k,\chi}.
\end{align*}
Observe that the monomials $y_{1,\chi} \cdots y_{k,\chi}$ and $y_{1,\chi'} \cdots y_{k,\chi'}$
have no common variables for any $\chi\neq \chi'\in \widehat{G}$.
By
\eqref{eq:annihilator-size}, we know that the number of nonzero summands is $|S^\perp|=[G:S]$. Therefore, the
displayed expression for $f_K(P^*\bm y)$ is a direct sum with $|S^\perp|=[G:S]$ nonzero monomial summands whenever $[G:S]\ge2$.
\end{proof}

Let \(G\) be a finite abelian group, $S$ a subgroup of $G$, and $k\ge 3$. We introduce a hypergraph $\widetilde{K}$ with vertex set
\[
V(\widetilde K)
=
\{(i,\chi):1\leq i\leq k,\ \chi\in \widehat G\},
\]
and edge set
\[
E(\widetilde K)
=
\{\{(1,\chi),\ldots,(k,\chi)\}:\chi\in S^\perp\}.
\]
Clearly, $\widetilde{K}$ is $k$-partite and $k$-uniform, with exactly $|S^\perp|$ pairwise disjoint edges and $k(|G|-|S^\perp|)$ isolated vertices.

\begin{theorem}[Fourier representation of adjacency tensor]\label{the-fourier-adj}
Retain the above notation. Suppose that $\mathcal{A}_K$ and $\mathcal{A}_{\widetilde{K}}$ are
the adjacency tensors of $K$ and $\widetilde{K}$, respectively.
Order the vertices of $K$ as $(1,g_1),\ldots,(1,g_{|G|}), \ldots, (k,g_1),\ldots,(k,g_{|G|})$, and the vertices of $\widetilde{K}$ as $(1,\chi_1)$, $\ldots$, $(1,\chi_{|\widehat{G}|})$, $\ldots$, $(k,\chi_1)$, $\ldots$, $(k,\chi_{|\widehat{G}|})$, where $g_1,\ldots,g_{|G|}$ are elements of $G$, and $\chi_1,\ldots,\chi_{|\widehat{G}|}$ are elements of $\widehat{G}$.
Then
\[c\mathcal{A}_{\widetilde{K}}=\overline P^{\,k}\mathcal{A}_K=P^k\mathcal{A}_K,\]
where $P=\operatorname{diag}(\underbrace{U,\ldots,U}_{k\text{ copies}})$, $ U = (|G|^{-\frac{1}{2}} \chi(g))_{\chi\in\widehat G,\;g\in G}$, and $c=|S|\,|G|^{\frac{k-2}{2}}$.
\end{theorem}

\begin{proof}
By the character orthogonality relation \eqref{eq:character-orthogonality},
$P$ is unitary. Let $\mathcal B$ denote the supersymmetric tensor obtained from $\mathcal A_K$ under the
invertible linear transformation $\bm y=P\bm x$, that is, $\mathcal A_K \bm{x}^k=\mathcal B {\bm y}^k$.
Since $\bm x=P^*\bm y$, the definition of the tensor action gives
\[
\mathcal B=(P^{-\top})^k\mathcal A_K=\overline P^{\,k}\mathcal A_K,
\]
where $P^{-\top}=(P^{-1})^\top$ and $\overline P$ is the entrywise complex conjugate of $P$.

Since $\bm x=P^*\bm y$, by Theorem~\ref{thm:direct-sum}, we get a polynomial
\[
f_K(P^*\bm y)
=c\sum_{\chi\in S^\perp}y_{1,\chi}\cdots y_{k,\chi}
\]
in variables $\{y_{i,\chi}:i\in[k],\ \chi\in S^\perp\}$.
By \eqref{eq:tensor-gradient}, we know that the tensor $\mathcal B$ satisfies
\[
\mathcal B\bm y^k=\mathcal A_K \bm{x}^k=k f_K(\bm x)=k f_K(P^*\bm y).
\]

On the other hand, by the definition of $\widetilde K$, we know that the hypergraph polynomial
\[f_{\widetilde{K}}(\bm y)=\sum_{\chi\in S^\perp}y_{1,\chi}\cdots y_{k,\chi}.\]
By \eqref{eq:tensor-gradient}, we have $\mathcal{A}_{\widetilde{K}}\bm y^k=kf_{\widetilde{K}}(\bm y)$.
It follows that
\[\mathcal{B}\y^k=k f_K(P^\ast\y)=ckf_{\widetilde{K}}(\y)= c\mathcal{A}_{\widetilde{K}}\y^k.\]
Since both tensors are supersymmetric, equality of their associated polynomials implies
\[
c\,\mathcal A_{\widetilde K}=\mathcal B=\overline P^{\,k}\mathcal A_K.
\]
Taking complex conjugates and using the fact that $\mathcal A_K$, $\mathcal A_{\widetilde K}$, and $c$ are real, we also obtain
$c\,\mathcal A_{\widetilde K}=P^k\mathcal A_K$.
\end{proof}

\begin{remark}
Observe that the Fourier matrix $U$ is unitary but need not be orthogonal, since finite abelian groups may have non-real-valued irreducible characters. A unitary transformation preserves the Hermitian form $\bm x^*\bm x$, but need not preserve the bilinear form $\bm x^\top\bm x$ used to normalize E-eigenvectors. Therefore, the decomposition in Theorem \ref{thm:direct-sum} cannot in general be used directly to preserve or determine the E-eigenvalues of $K$.
\end{remark}

\begin{example}\label{ex:direct-sum}
Consider $G=\Z_3\times\Z_3$, and let $S=\{(0,0),(1,0),(2,0)\}$ be the subgroup of order $3$. Take $k=3$, and write $K=3\mbox{-}\textrm{Cay}(G,S)$. The
hypergraph polynomial of $K$ is
\begin{align*}
f_K(\bm x)=& \sum_{g_1+g_2+g_3\in S}x_{1,g_1}x_{2,g_2}x_{3,g_3}\\
=& \sum_{\substack{g_1+g_2+g_3=(0,0)}}x_{1,g_1}x_{2,g_2}x_{3,g_3}
+
\sum_{\substack{g_1+g_2+g_3=(1,0)}}x_{1,g_1}x_{2,g_2}x_{3,g_3}+
\sum_{\substack{g_1+g_2+g_3=(2,0)}}x_{1,g_1}x_{2,g_2}x_{3,g_3}.
\end{align*}

The character group of $G=\Z_3\times\Z_3$ satisfies $\widehat G\cong\Z_3\times\Z_3$. To be precise, the irreducible character $\chi_{(\alpha,\beta)}$ indexed by $(\alpha,\beta)\in\Z_3\times \Z_3$ is given by
\[
\chi_{(\alpha,\beta)}(g)=\omega^{\alpha u+\beta v}
\]
for any $g=(u, v)\in G=\Z_3\times \Z_3$, where $\omega=e^{\frac{2\pi \textrm{i}}{3}}$, and $\textrm{i}=\sqrt{-1}$.
The annihilator of $S$ is
\[
S^\perp=\{\chi_{(\alpha,\beta)}\in\widehat G:\chi_{(\alpha,\beta)}(s)=1\text{ for all }s\in S\}.
\]
It follows that $\chi_{(\alpha, \beta)}\in S^\perp$ if and only if $\alpha=0\in \Z_3$. Therefore, we get
\[
S^\perp=\{\chi_{(0,0)},\chi_{(0,1)},\chi_{(0,2)}\},
\]
and $|S^\perp|=3=[G:S]$.
By Theorem \ref{thm:fourier-decomposition}, we have
\begin{equation*}
    f_K(\bm{x})=\frac{|S|}{|G|}\sum_{\chi\in S^\perp}
L_{1,\chi}L_{2,\chi}L_{3,\chi}=\frac13\sum_{\chi\in S^\perp}
L_{1,\chi}L_{2,\chi}L_{3,\chi}
\end{equation*}
with $L_{i,\chi}=\sum_{g\in G}\chi(g)x_{i,g}$ for each $\chi\in \widehat{G}$.
Using the normalized Fourier coordinates
\[
y_{i,\chi}=\frac1{\sqrt{|G|}}\sum_{g\in G}\chi(g)x_{i,g}
=\frac13\sum_{g\in G}\chi(g)x_{i,g},
\quad i=1,2,3,
\]
we have $L_{i,\chi}=3y_{i,\chi}$. By Theorem \ref{thm:direct-sum} and $c=|S| \sqrt{|G|}=9$, the polynomial $f_K(\bm x)$ is a direct sum given by
\begin{align*}
f_K(P^\ast \bm y)=& \ c\sum_{\chi\in S^\perp}y_{1,\chi}y_{2,\chi}y_{3,\chi}\\
=& \ 9\left(
y_{1,\chi_{(0,0)}}y_{2,\chi_{(0,0)}}y_{3,\chi_{(0,0)}}
+
y_{1,\chi_{(0,1)}}y_{2,\chi_{(0,1)}}y_{3,\chi_{(0,1)}}
+
y_{1,\chi_{(0,2)}}y_{2,\chi_{(0,2)}}y_{3,\chi_{(0,2)}}
\right).
\end{align*}
\end{example}

\section{Orthogonal Fourier decomposition and E-eigenvalues}\label{sec:eigenvalues}

In this section, we determine the E-eigenvalues of the $k$-partite Cayley hypergraph associated with a finite abelian group $G$ and a subgroup $S$ such that $G/S$ is an elementary abelian $2$-group.
Recall that a finite abelian group is an \emph{elementary abelian $2$-group}
if every nonzero element has order $2$; equivalently, it is isomorphic to
$\Z_2^d$ for some $d\ge0$, including the trivial group when $d=0$.
Thus, our assumption is equivalent to
\[
2G\subseteq S,\quad 2G:=\{2g:g\in G\}.
\]
Throughout this section, let $k\ge3$, $q=|G|$, and $K=\kCay(G,S)$, and assume $2G\subseteq S$.
Observe that the character group $\widehat G$ is isomorphic to $G$.
Fix such an isomorphism and denote by $\chi_g\in\widehat G$ the character corresponding to $g\in G$.

For each $\chi_g\in S^\perp$ and $h\in G$, we have
\[
\chi_g(h)^2=\chi_g(2h)=1,
\]
since $2h\in S$. Hence the normalized character rows
\[
u_{\chi_g}=\bigl(q^{-1/2}\chi_g(h)\bigr)_{h\in G},
\quad \chi_g\in S^\perp,
\]
are real and, by \eqref{eq:character-orthogonality}, form an orthonormal set in $\R^q$.
Complete these rows to an orthonormal basis of $\R^q$, and index the added
rows by $\widehat G\setminus S^\perp$.
For $\chi_g\notin S^\perp$, the symbol $\chi_g$ is only a label for an added
row $u_{\chi_g}$; this row need not be the normalized character row.
Let
\[
R=\bigl(u_{\chi_g}(h)\bigr)_{\chi_g\in\widehat G,\ h\in G},
\]
so that $R$ is a real orthogonal matrix.
For each $i\in[k]$, write
\[
\bm x_i=(x_{i,h})_{h\in G},\quad
\bm y_i=(y_{i,\chi_g})_{\chi_g\in\widehat G},\quad
\bm y_i=R\bm x_i.
\]
In particular, the coordinates indexed by $S^\perp$ are the normalized Fourier coordinates
\begin{equation}\label{eq:real-fourier-coordinates}
y_{i,\chi_g}=\frac1{\sqrt q}\sum_{h\in G}\chi_g(h)x_{i,h},
\quad i\in[k],\quad \chi_g\in S^\perp.
\end{equation}
For all $k$ vertex parts, set
\[
\bm y=Q\bm x,\quad
Q=\operatorname{diag}(\underbrace{R,\ldots,R}_{k\text{ copies}}).
\]
Then $Q$ is real and orthogonal, and $\bm x=Q^\top\bm y$.
In general, $R$ need not equal the full Fourier matrix $U$ from
Section~\ref{sec:fourier}; only the rows indexed by $S^\perp$ are prescribed.

Set
\[
r=|S^\perp|=[G:S],\quad
c=|S|q^{\frac{k-2}{2}}.
\]
By the definition of $\widetilde{K}$, its hypergraph polynomial is
\begin{equation}\label{eq:binary-direct-sum}
  f_{\widetilde{K}}(\bm y)=\sum_{\chi_g \in S^\perp}y_{1, \chi_g}y_{2, \chi_g}\cdots y_{k, \chi_g}.
\end{equation}
Clearly, the variable $y_{i, \chi_g}$ with $\chi_g \notin S^\perp$ does not occur in $f_{\widetilde{K}}(\bm y)$.

\begin{lemma}[Orthogonal Fourier reduction]
\label{thm:orthogonal-reduction}
Let $G$ be a finite abelian group and $S$ a subgroup such that $G/S$ is an elementary abelian $2$-group.
For the real orthogonal matrix $Q$ constructed above, we have
\[
f_K(Q^\top\bm y)=c f_{\widetilde K}(\bm y),\quad
Q^k\mathcal A_K=c\mathcal A_{\widetilde K}.
\]
Consequently,
\begin{align*}
\operatorname{Spec}_E(\mathcal A_K)
&=
c\,\operatorname{Spec}_E(\mathcal A_{\widetilde K}).
\end{align*}
The same identity holds with $\operatorname{Spec}_E$ replaced by $\operatorname{Spec}_Z$.
\end{lemma}

\begin{proof}
For $\chi_g\in S^\perp$, equation~\eqref{eq:real-fourier-coordinates} gives
$L_{i,\chi_g}=\sqrt q\,y_{i,\chi_g}$.
By Theorem~\ref{thm:fourier-decomposition},
\[
f_K(Q^\top\bm y)
=\frac{|S|}{q}\sum_{\chi_g\in S^\perp}\prod_{i=1}^k
  \bigl(\sqrt q\,y_{i,\chi_g}\bigr)
=c f_{\widetilde K}(\bm y).
\]
By the definition of the tensor action and \eqref{eq:tensor-gradient},
\[
(Q^k\mathcal A_K)\bm y^k
=\mathcal A_K(Q^\top\bm y)^k
=k f_K(Q^\top\bm y)
=c\mathcal A_{\widetilde K}\bm y^k.
\]
Since both tensors are supersymmetric, equality of their associated
polynomials implies $Q^k\mathcal A_K=c\mathcal A_{\widetilde K}$.
As $Q$ is real and orthogonal, Theorem~\ref{thm:qi-orthogonal-invariance} gives
\[
\operatorname{Spec}_E(\mathcal A_K)
=c\operatorname{Spec}_E(\mathcal A_{\widetilde K}).
\]
Since $Q$ is real, the correspondence between eigenvectors also preserves real vectors. Thus the same argument applies to Z-eigenvalues.
\end{proof}

\begin{theorem}[Complex E-eigenvalues]\label{thm:complex-e-spectrum}
Let $G$ be a finite abelian group, $S$ a subgroup such that $G/S$ is an elementary abelian $2$-group, and $k\geq3$. Write $K=\kCay(G,S)$, $q=|G|$, $r=|S^\perp|=[G:S]$, and $c=|S|q^{(k-2)/2}$. For \(1\le s\le r\), define
\[
\Sigma_s=\left\{\sum_{i=1}^s\omega_i:\omega_i^{k-2}=1,\
\sum_{i=1}^s\omega_i\neq0\right\}.
\]
Then the complex E-eigenvalues of $\mathcal A_K$ are
\begin{align}\label{eq:complex-e-spectrum}
  \operatorname{Spec}_E(\mathcal A_{K})
  =c\operatorname{Spec}_E(\mathcal A_{\widetilde K})
  =\{0\}\cup
  \bigcup_{s=1}^r
  \left\{
    \lambda\in\C :
    \lambda^2=c^2(k\zeta)^{-(k-2)}, \zeta\in\Sigma_s
  \right\}.
\end{align}
Equivalently, since $|S|=q/r$, one may replace $c^2$ by $q^k/r^2$.
\end{theorem}

\begin{proof}
By Lemma~\ref{thm:orthogonal-reduction}, we know
$\operatorname{Spec}_E(\mathcal A_K)=c\operatorname{Spec}_E(\mathcal A_{\widetilde K})$.
Thus, it suffices to determine the E-eigenvalues of
$\mathcal A_{\widetilde K}$. By \eqref{eq:eigenvalue-lambda}, the E-eigenvalues of
$\mathcal A_{\widetilde K}$ are given by
\begin{align}\label{eq-4.2-E-eigenvalue of A tildeK}
\nabla f_{\widetilde K}(\bm y)=\lambda\bm y,
\quad
\bm y^\top\bm y=1,
\end{align}
where
\[ f_{\widetilde K}(\bm y)=\sum_{\chi_g \in S^\perp}y_{1,\chi_g}y_{2,\chi_g}\cdots y_{k,\chi_g}.\]
To be precise, equations \eqref{eq-4.2-E-eigenvalue of A tildeK} read as
\[\dfrac{\partial}{\partial y_{i,\chi_g}} f_{\widetilde{K}}(\bm y)=\lambda y_{i,\chi_g}\]
for all $i\in [k]$ and $\chi_g \in \widehat G$.
If $\chi_g \in S^\perp$, then we have
\begin{equation}\label{eq:block-equations-complex}
\prod_{\ell\neq i}y_{\ell,\chi_g}
=\lambda y_{i,\chi_g},
\quad i \in [k].
\end{equation}
If $\chi_g \notin S^\perp$, then we have
\begin{equation}\label{eq:null-equations-complex}
0=\lambda y_{i,\chi_g}
\end{equation}
since the variable $y_{i,\chi_g}$ does not appear in $f_{\widetilde K}(\bm y)$.

Assume that $\lambda\neq 0$ is an E-eigenvalue of $\mathcal A_{\widetilde{K}}$. Then
\eqref{eq:null-equations-complex} implies
$y_{i,\chi_g}=0$ for all $\chi_g \notin S^\perp$.
Next, we fix some $\chi_g \in S^\perp$ in equation \eqref{eq:block-equations-complex}. If $y_{i,\chi_g}\neq 0$ and $y_{j,\chi_g}=0$ for some $j\neq i$, then
\[
\prod_{\ell\neq i}y_{\ell,\chi_g}=0
\quad {\rm and}\quad
\lambda y_{i,\chi_g}\neq 0,
\]
which is impossible. Hence $y_{1,\chi_g},\ldots,y_{k,\chi_g}$ are either
all zero or all nonzero.

Let $B=\{\chi_g \in S^\perp :y_{1,\chi_g}\neq 0\}$, and set $s=|B|$. Since $\bm y\neq0$,
we have $1\leq s\leq r$. For $\chi_g \in B$, we define $P_{\chi_g}$ by
\[
P_{\chi_g}=\prod_{j=1}^k y_{j,\chi_g}.
\]
Multiplying \eqref{eq:block-equations-complex} by $y_{i,\chi_g}$, we obtain
$P_{\chi_g}=\lambda y_{i,\chi_g}^2$, so $y_{i,\chi_g}^2$ is independent of $i$.
Hence
\[
\lambda^2
=
\frac{P_{\chi_g}^2}{y_{1,\chi_g}^4}
=
\frac{\prod_{i=1}^k y_{i,\chi_g}^2}{y_{1,\chi_g}^4}
=
\frac{y_{1,\chi_g}^{2k}}{y_{1,\chi_g}^4}
=
y_{1,\chi_g}^{2(k-2)}.
\]
Thus, $y_{1,\chi_g}^{2(k-2)}$ is the same for all $\chi_g\in B$.

Choose $\rho\in\C\setminus\{0\}$ such that $\rho^{k-2}=\lambda^2$.
For each $\chi_g\in B$, write \(y_{i,\chi_g}^2=\rho\omega_{\chi_g}\); then
\(\omega_{\chi_g}^{k-2}=1\). The normalization condition gives
\[
1
=
\bm y^\top\bm y
=
\sum_{{\chi_g}\in B}\sum_{i=1}^k y_{i,{\chi_g}}^2
=
k\rho\sum_{{\chi_g}\in B}\omega_{\chi_g}.
\]
Thus, we have
\[
\zeta:=\sum_{{\chi_g}\in B}\omega_{\chi_g}\neq0,
\quad
\rho=\frac{1}{k\zeta}.
\]
Consequently, the E-eigenvalues of $\mathcal A_{\widetilde K}$ satisfy
\[
\lambda^2=\rho^{k-2}=(k\zeta)^{-(k-2)},\quad \zeta\in\Sigma_s.
\]

Conversely, let $B\subseteq S^\perp$ with $|B|=s$, and choose $\omega_{\chi_g}$
with $\omega_{\chi_g}^{k-2}=1$ such that
\[
\zeta=\sum_{{\chi_g}\in B}\omega_{\chi_g}\neq0.
\]
Set $\rho=(k\zeta)^{-1}$ and let $\lambda$ be either solution of $\lambda^2=\rho^{k-2}$.
For each $\chi_g\in B$, choose $t_{\chi_g}\in\C$ such that $t_{\chi_g}^2=\rho\omega_{\chi_g}$. Since
\[
t_{\chi_g}^{2(k-2)}
=
\rho^{k-2}\omega_{\chi_g}^{k-2}
=
\rho^{k-2}
=
\lambda^2,
\]
we have
\[
\left(\frac{\lambda}{t_{\chi_g}^{k-2}}\right)^2=1.
\]
We may therefore choose $\varepsilon_{i,{\chi_g}}\in\{\pm1\}$ such that
\[
\prod_{i=1}^k\varepsilon_{i,{\chi_g}}
=
\frac{\lambda}{t_{\chi_g}^{k-2}}.
\]
Define $y_{i,{\chi_g}}=\varepsilon_{i,{\chi_g}}t_{\chi_g}$ for $\chi_g \in B$ and $i\in [k]$, and set all remaining coordinates to zero. Then
\[
\bm y^\top\bm y
=
k\sum_{{\chi_g}\in B}y_{1,{\chi_g}}^2
=
k\rho\zeta
=
1.
\]
For each $\chi_g\in B$,
\[
\prod_{\ell\ne i}y_{\ell,\chi_g}
=\frac{\prod_{\ell=1}^k\varepsilon_{\ell,\chi_g}}{\varepsilon_{i,\chi_g}}
 t_{\chi_g}^{k-1}
=\lambda\varepsilon_{i,\chi_g}t_{\chi_g}
=\lambda y_{i,\chi_g}.
\]
All other equations are satisfied because their coordinates are zero.
Thus $\lambda$ is an E-eigenvalue of $\mathcal A_{\widetilde K}$, and multiplying by $c$ realizes every nonzero value on the right-hand side of \eqref{eq:complex-e-spectrum}.

It remains to show that zero is an E-eigenvalue. Take
${\chi_g}\in S^\perp$ and set
\[
y_{i,\chi_{g'}}=
\begin{cases}
1, & (i,\chi_{g'})=(1,{\chi_g}),\\
0, & (i,\chi_{g'})\neq(1,{\chi_g}).
\end{cases}
\]
Then $\bm y^\top\bm y=1$.
For $\chi_{g'}={\chi_g}$,
\[
\dfrac{\partial}{\partial y_{i,{\chi_g}}} f_{\widetilde{K}}(\bm y)
=
\prod_{j\neq i}y_{j,{\chi_g}}
=
0,
\]
since $k\ge 3$ and at most one $y_{j,{\chi_g}}$ is nonzero. For
$\chi_{g'} \neq {\chi_g}$, $\dfrac{\partial}{\partial y_{i,\chi_{g'}}} f_{\widetilde{K}}(\bm y)=0$.
Thus
\[
\nabla f_{\widetilde{K}}(\bm{y})=0=\lambda\bm y,
\]
so $\lambda=0$ is an E-eigenvalue of $\mathcal A_{\widetilde K}$ and hence of $\mathcal A_K$.

Combining the above arguments, we obtain
\begin{align*}
  \operatorname{Spec}_E(\mathcal A_{K})
  =c\operatorname{Spec}_E(\mathcal A_{\widetilde K})
  =\{0\}\cup
  \bigcup_{s=1}^r
  \left\{
    \lambda\in\C :
    \lambda^2=c^2(k\zeta)^{-(k-2)}, \zeta\in\Sigma_s
  \right\}.
\end{align*}
\end{proof}

\begin{corollary}[Eigenvalues with real eigenvectors]\label{cor:z-spectrum}
Under the hypotheses of Theorem~\ref{thm:complex-e-spectrum}, the Z-eigenvalues of $\mathcal A_{K}$ are
\begin{equation}\label{eq:z-spectrum}
  \operatorname{Spec}_Z(\mathcal A_{K})
  =\{0\}\cup \bigcup_{s=1}^{r}
  \left\{
    \pm |S|q^{\frac{k-2}{2}}(sk)^{-\frac{k-2}{2}}
  \right\}.
\end{equation}
Equivalently,
\[
  \operatorname{Spec}_Z(\mathcal A_{K})
  =\{0\}\cup \bigcup_{s=1}^{r}
  \left\{
    \pm \frac{q^{\frac{k}{2}}}{r}(sk)^{-\frac{k-2}{2}}
  \right\}.
\]
\end{corollary}

\begin{proof}
By Lemma~\ref{thm:orthogonal-reduction}, it suffices to determine the Z-eigenvalues of $\mathcal A_{\widetilde K}$ and then multiply them by $c$.
Let $\lambda\ne0$ be a Z-eigenvalue of $\mathcal A_{\widetilde K}$ with real normalized eigenvector $\bm y$.
Put
\[
B=\{\chi_g\in S^\perp:y_{1,\chi_g}\ne0\}.
\]
The proof of Theorem~\ref{thm:complex-e-spectrum} shows that
$1\le |B|\le r$, that $y_{i,\chi_g}=0$ whenever
$i\in[k]$ and $\chi_g\notin B$, and that
\[
y_{i,\chi_g}^2=y_{1,\chi_g}^2>0,
\quad
(y_{1,\chi_g}^2)^{k-2}=\lambda^2
\]
for every $\chi_g\in B$ and $i\in[k]$.
Since the map $a\mapsto a^{k-2}$ is injective on the positive real numbers, these squares have a common value $a$. Therefore,
\[
1=\sum_{\chi_g\in B}\sum_{i=1}^k y_{i,\chi_g}^2=k|B|a,
\]
and hence
\[
\lambda^2=a^{k-2}=(k|B|)^{-(k-2)}.
\]
Conversely, fix an integer $s$ with $1\le s\le r$ and choose a subset
$B\subseteq S^\perp$ with $s$ elements. For either sign
$\delta\in\{\pm1\}$, choose $\varepsilon_{i,\chi_g}\in\{\pm1\}$ such that
$\prod_{i=1}^k\varepsilon_{i,\chi_g}=\delta$ for every $\chi_g\in B$.
Set $y_{i,\chi_g}=\varepsilon_{i,\chi_g}/\sqrt{sk}$ for
$\chi_g\in B$ and $i\in[k]$, and set all remaining coordinates to zero.
This gives a real normalized eigenvector with eigenvalue $\delta(sk)^{-(k-2)/2}$.
The zero eigenvalue is obtained by the real vector constructed in the proof of
Theorem~\ref{thm:complex-e-spectrum}. Multiplying by
$c=|S|q^{\frac{k-2}{2}}$, we obtain \eqref{eq:z-spectrum}.
\end{proof}

\pagebreak[3]
\begin{corollary}[Special cases]\label{cor:special-cases}
Retain the hypotheses and notation of Theorem~\ref{thm:complex-e-spectrum}.
\begin{enumerate}[label=\textup{(\arabic*)}]
  \item If $S=\{0\}$, then $r=q$ and the Z-eigenvalues of $\mathcal A_{K}$ are
  \[
   \operatorname{Spec}_Z(\mathcal A_K)
= \{0\}\cup \bigcup_{s=1}^{q}
    \left\{
      \pm q^{\frac{k-2}{2}}(sk)^{-\frac{k-2}{2}}
    \right\}.
  \]
  \item If $k=3$, then the complex $E$-eigenvalues and the $Z$-eigenvalues of $\mathcal A_{K}$ coincide, that is,
\[
\operatorname{Spec}_E(\mathcal A_K)
=
\operatorname{Spec}_Z(\mathcal A_K)
=
\{0\}
\cup
\bigcup_{s=1}^{r}
\left\{
\pm\frac{|S|\sqrt{q}}{\sqrt{3s}}
\right\}.
\]
  \item If $G=\Z_2^m$, then for every subgroup $S$ of $G$, the complex
  E-eigenvalues and Z-eigenvalues are given by
  \eqref{eq:complex-e-spectrum} and \eqref{eq:z-spectrum}, respectively,
  with $q=2^m$. In this case, using the standard character indexing
  $\chi_g(h)=(-1)^{g\cdot h}$, where $g\cdot h$ is computed in $\Z_2$,
  one may take
  \[
  R=U=\bigl(2^{-m/2}(-1)^{g\cdot h}\bigr)_{g,h\in\Z_2^m},
  \]
  the normalized Walsh--Hadamard matrix. Thus $Q=P$ in this special case.
\end{enumerate}
\end{corollary}

\begin{proof}
Parts \textup{(1)} and \textup{(2)} follow directly from
Theorem~\ref{thm:complex-e-spectrum} and Corollary~\ref{cor:z-spectrum}.
For \textup{(3)}, we have $2G=\{0\}\subseteq S$, so the hypotheses hold
for every subgroup $S$. All characters of $G$ are real-valued, and their
normalized rows already form a real orthonormal basis. Hence the full
Walsh--Hadamard matrix may be used as $R$.
\end{proof}

\begin{example}\label{ex:direct-sum-2}
Consider $G=\Z_2^3$ and the subgroup
$S=\{(0,0,0),(0,0,1)\}$ of order $2$. Take $k=4$,
and write $K=4\mbox{-}\textrm{Cay}(G,S)$. Then
\[
q=|G|=8,\quad |S|=2,\quad |G/S|=\frac{|G|}{|S|}=4,
\]
and $c=|S|q^{\frac{k-2}{2}}=16$. The hypergraph
polynomial of \(K\) is
\[
f_K(\bm x)=\sum_{g_1+g_2+g_3+g_4\in S}
x_{1,g_1}x_{2,g_2}x_{3,g_3}x_{4,g_4}.
\]
Here $U$ is the $8\times8$ normalized Walsh--Hadamard matrix, and
$Q=\operatorname{diag}(U,U,U,U)\in\R^{32\times32}$ is the real
orthogonal matrix in Corollary~\ref{cor:special-cases}\textup{(3)}.
Since $\bm y=Q\bm x$, Lemma~\ref{thm:orthogonal-reduction} gives
\[
f_K(Q^\top\bm y)=16\sum_{{\chi_g}\in S^\perp}y_{1,{\chi_g}}y_{2,{\chi_g}}y_{3,{\chi_g}}y_{4,{\chi_g}}
\]
where $|S^\perp|=[G:S]=4$. Only the $16$ coordinates
$y_{i,\chi_g}$ with $i\in[4]$ and $\chi_g\in S^\perp$ occur in the
transformed polynomial. It is easy to check that zero is an
E-eigenvalue. We now determine the nonzero E-eigenvalues. Let
$B=\{{\chi_g}\in S^\perp:y_{1,{\chi_g}}\neq0\}$ and $s=|B|$. For nonzero E-eigenvalues,  by \eqref{eq:eigenvalue-lambda} the equations
give
\begin{equation}\label{eq:example-block}
16\prod_{j\neq i}y_{j,{\chi_g}}=\lambda y_{i,{\chi_g}},
\quad i\in[4].
\end{equation}
For each $\chi_g\in B$, the coordinates $y_{i,\chi_g}$, $i\in[4]$, are nonzero, and $y_{i,{\chi_g}}^2$ is independent of $i$. Denote its common value by $\rho_{\chi_g}$, so that $y_{i,{\chi_g}}^2=\rho_{\chi_g}$ for $i\in[4]$.
Multiplying \eqref{eq:example-block} by $y_{i,\chi_g}$ gives
$16P_{\chi_g}=\lambda\rho_{\chi_g}$, where $P_{\chi_g}=\prod_{j=1}^4y_{j,\chi_g}$.
Since $P_{\chi_g}^2=\rho_{\chi_g}^4$, we obtain
$\lambda^2=16^2\rho_{\chi_g}^2$.
Thus $\rho_{\chi_g}^2$ is independent of $\chi_g\in B$. Hence there exist $\rho\in\C\setminus\{0\}$ and $\omega_{\chi_g}\in\{\pm1\}$ such that
\begin{equation*}
  y_{i,{\chi_g}}^2=\rho\omega_{\chi_g},
\quad i\in[4],\ \chi_g\in B.
\end{equation*}
Then the
normalization condition gives
\[
1=\sum_{{\chi_g}\in B}\sum_{i=1}^4 y_{i,{\chi_g}}^2=4\rho\sum_{\chi_g\in B}\omega_{\chi_g}.
\]
Set $\zeta=\sum_{\chi_g\in B}\omega_{\chi_g}$. Then $\zeta\neq0$ and $\rho=\frac{1}{4\zeta}$.
Consequently,
\begin{equation*}
 \lambda^2
=16^2\rho^2
=\frac{16}{\zeta^2}.
\end{equation*}
As $s$ ranges from $1$ to $4$ and $\omega_{\chi_g}\in\{\pm1\}$, the possible nonzero sums are precisely
\[
\zeta=\pm j,\quad 1\le j\le 4.
\]
Therefore
\[
\lambda^2=\frac{16}{j^2},
\quad
\lambda=\pm\frac{4}{j},
\quad j \in[4].
\]
This agrees with Theorem \ref{thm:complex-e-spectrum}, which gives
$\lambda^2=c^2(k\zeta)^{-(k-2)}$ with $k-2=2$ and
$\zeta\in\Sigma_s$. Hence
\[
\operatorname{Spec}_E(\mathcal A_K)
=
\left\{
0,\ \pm 4,\ \pm 2,\ \pm\frac{4}{3},\ \pm 1
\right\}.
\]
\end{example}

\section*{Acknowledgements}

\begingroup
\setlength{\emergencystretch}{1em}
 Y.-H. Bao was partially supported by the National Natural Science Foundation of China (No.~12371015) and the Excellent University Research and Innovation Team in Anhui Province (No.~2024AH010002).
 Y.-Z. Fan was supported by the National Natural Science Foundation of China (No.~12331012).
\par
\endgroup

\section*{Conflict of Interest}
The authors declare no conflicts of interest.


\begin{thebibliography}{99}

\bibitem{Babai1979}
L. Babai.
Spectra of Cayley graphs.
\textit{Journal of Combinatorial Theory, Series B}, 27: 180--189, 1979.

\bibitem{Buratti1994}
M. Buratti.
Cayley, Marty and Schreier hypergraphs.
\textit{Abhandlungen aus dem Mathematischen Seminar der Universit\"at Hamburg}, 64: 151--162, 1994.

\bibitem{Biggs1993}
N. Biggs.
\textit{Algebraic Graph Theory}, 2nd ed.
Cambridge University Press, Cambridge, 1993.


\bibitem{CooperDutle2012}
J. Cooper and A. Dutle.
Spectra of uniform hypergraphs.
\textit{Linear Algebra and its Applications}, 436(9): 3268--3292, 2012.

\bibitem{FangHuangHuang2026}
L. Fang, H.-L. Huang and J. Huang.
Simultaneous block diagonalization of symmetric matrices via congruence.
\textit{Linear Algebra and its Applications}, 737: 193--212, 2026.

\bibitem{FangHuangLiao2025}
L. Fang, H.-L. Huang and L. Liao.
Simultaneous direct sum decompositions of several multivariate polynomials.
\textit{Linear Algebra and its Applications}, 724: 320--335, 2025.

\bibitem{GuanChuChu2018}
Y. Guan, M. T. Chu and D.-L. Chu,
SVD-based algorithms for the best rank-1 approximation of a symmetric tensor,
\emph{SIAM Journal on Matrix Analysis and Applications},
39(3): 1095--1115, 2018.

\bibitem{GodsilRoyle2001}
C. Godsil and G. Royle.
\textit{Algebraic Graph Theory}.
Springer, New York, 2001.

\bibitem{Harrison1975}
D. K. Harrison.
A Grothendieck ring of higher degree forms.
\textit{Journal of Algebra}, 35(1): 123--138, 1975.


\bibitem{HuangLvZhangYe2018}
H.-L. Huang, H. Lu, Y. Ye and C. Zhang.
On centres and direct sum decompositions of higher degree forms.
\textit{Linear and Multilinear Algebra}, 70(22): 7290--7306, 2022.

\bibitem{HuangLuYeZhang2021}
H.-L. Huang, H. Lu, Y. Ye and C. Zhang.
Diagonalizable higher degree forms and symmetric tensors.
\textit{Linear Algebra and its Applications}, 613: 151--169, 2021.

\bibitem{KofidisRegalia2002}
E. Kofidis and P. A. Regalia.
On the best rank-1 approximation of higher-order supersymmetric tensors.
\textit{SIAM Journal on Matrix Analysis and Applications}, 23(3): 863--884, 2002.

\bibitem{Luong2009}
B. Luong.
\textit{Fourier Analysis on Finite Abelian Groups}.
Birkh\"auser, Boston, 2009.

\bibitem{Lim2005}
L.-H. Lim.
Singular values and eigenvalues of tensors: a variational approach.
In \textit{Proceedings of the IEEE International Workshop on Computational Advances in Multi-Sensor Adaptive Processing}, pages 129--132, 2005.

\bibitem{LiQiYu2013}
G.-Y. Li, L.-Q. Qi and G.-H. Yu,
The Z-eigenvalues of a symmetric tensor and its application to spectral
hypergraph theory,
\emph{Numerical Linear Algebra with Applications},
20(6): 1001--1029, 2013.

\bibitem{LuWangXu2006}
Z.-P. Lu, C.-Q. Wang and M.-Y. Xu.
Semisymmetric cubic graphs constructed from Bi-Cayley graphs of \(A_n\).
\textit{Ars Combinatoria}, 80: 177--187, 2006.


\bibitem{NiQiWangWang2007}
G.-Y. Ni, L.-Q. Qi, F. Wang and Y.-J. Wang,
The degree of the E-characteristic polynomial of an even order tensor,
\emph{Journal of Mathematical Analysis and Applications},
329(2): 1218--1229, 2007.

\bibitem{PearsonZhang2014}
K. J. Pearson and T. Zhang.
On spectral hypergraph theory of the adjacency tensor.
\textit{Graphs and Combinatorics}, 30(5): 1233--1248, 2014.

\bibitem{Qi2005}
L. Qi.
Eigenvalues of a real supersymmetric tensor.
\textit{Journal of Symbolic Computation}, 40(6): 1302--1324, 2005.

\bibitem{Steinberg2012}
B. Steinberg.
\textit{Representation Theory of Finite Groups: An Introductory Approach}.
Springer, 2012.

\bibitem{Terras1999}
A. Terras.
\textit{Fourier Analysis on Finite Groups and Applications}.
Cambridge University Press, Cambridge, 1999.


\bibitem{XiongJiangLiuQin2022}
L. Xiong, Z.-F. Jiang, J.-Z. Liu and Q. Qin,
New Z-eigenvalue localization set for tensor and its application in
entanglement of multipartite quantum states,
\emph{Mathematics},
10(15): 2624, 2022.

\bibitem{ZouMeng2007}
H. Zou and J.-X. Meng.
Some algebraic properties of Bi-Cayley graphs.
\textit{Acta Mathematica Sinica, Chinese Series}, 50(5): 1075--1080, 2007.
\end{thebibliography}
\end{document}